\documentclass[12pt,a4paper]{amsart}
\usepackage{amsthm, amssymb}
\usepackage{amsfonts}
\usepackage{amsthm}
\usepackage{amsmath}
\usepackage{amscd}
\usepackage[latin2]{inputenc}
\usepackage{t1enc}
\usepackage[mathscr]{eucal}

\numberwithin{equation}{section}
\usepackage[margin=2.9cm]{geometry}
\usepackage{epstopdf} 
\usepackage{hyperref}
\newtheorem{proposition}{Proposition}[section]

 \newtheorem{definition}{Definition}[section]
\newtheorem{remark}{Remark}[section]
\theoremstyle{definition}
\newtheorem{theorem}{Theorem}[section]
\theoremstyle{plain}

\newtheorem{lemma}{Lemma}[section]

 \theoremstyle{definition}

\title[]{Partial classification of irreducible modules for loop-Witt algebras}

\author[]{Priyanshu Chakraborty AND S. Eswara Rao }

\address{Priyanshu Chakraborty, Harish-Chandra Research Institute (HBNI), Chhatnag Road, Jhunsi, Prayagraj(Allahabad) 211019, Uttar Pradesh, India.}
 
\email{priyanshuchakraborty@hri.res.in}

\address{S. Eswara Rao, School of Mathematics, Tata Institute of Fundamental Research, Homi Bhabha Road, Colaba, Mumbai 400005, India.}

\email{sena98672@gmail.com, senapati@math.tifr.res.in}

\begin{document}

\begin{abstract} 
Consider the Lie algebra of the group of diffeomorphisms of a $n$-dimensional torus which is also known as the derivation algebra of the Laurent polynomial algebra $A$ over $n$ commuting variables, denoted by $DerA$. In this paper we consider the Lie algebra $(A\rtimes DerA)\otimes B$ for some commutative associative finitely generated unital algebra $B$ over $\mathbb C$ and classify all irreducible modules for $(A\rtimes DerA)\otimes B$ with finite dimensional weight spaces.  In particularly, we show that Larsson's constructed modules of tensor fields exhausts all irreducible modules for $(A\rtimes DerA)\otimes B$.
\end{abstract}

\maketitle

\section{Introduction} 
Importance of the representation theory of infinite dimensional Lie algebras are well known in both mathematics and physics. Many mathematicians studied representations for the Lie algebra of the derivation algebra of $A=\mathbb C[t_1^{\pm 1}, t_2^{\pm 1},...,t_n^{\pm 1}]$, so called Witt algebra denoted by $W_n$ or $DerA$, for reference \cite{1,10,11}. In particular, for $n=1$ the universal central extension for the derivation algebra of $A$ is known as the Virasoro algebra, generally denoted by $Vir$. Representations of Virasoro algebras were well studied in \cite{15,16,17} and references therein. In \cite{15}, O.Mathieu proved that a simple module for $Vir$ with finite dimensional weight spaces is either a highest weight module, a lowest weight module or a uniformly bounded module of intermediate series. Moreover from an interesting paper (\cite{12}) this result comes out that $DerA$ has no non-trivial central extensions for $n \geq 2$. In \cite{13}, Larsson constructed a large class of modules for $DerA$, so called modules of tensor fields. In \cite{1}, the second author proved that for $n \geq 2$, modules of tensor fields for $DerA$ exhausts all irreducible weight modules for $A\rtimes DerA$ with finite dimensional weight spaces. In \cite{10}, complete classification of irreducible modules for $DerA$ has been done by Yuly Billig and V. Futorny, which was conjectured by the second author. \\

On the other hand, in present days mathematicians are showing interest to study representations of loop algebras of some well known  Lie algebras. For instance, representation of current algebras (also called loop algebras) for finite dimensional simple Lie algebras has been studied in \cite{8,19}. In \cite{7}, representation of current algebras were generalized for finite dimensional perfect Lie algebras. In \cite{18}, loop Affine Lie algebras were studied, in fact any irreducible module for the well known Toroidal Lie algebras has been reduced to an irreducible module for the loop Affine Lie algebras. Further, in \cite{6,9} modules for loop Affine-Virasoro algebras and in \cite{14} modules for loop toroidal Lie algebras has been studied. For $B=\mathbb C[t^{\pm 1}]$, complete classification of irreducible $Vir\otimes B$-modules has been done in \cite{2}. Later this work was generalized by A.Savage for $Vir \otimes B$, where $B$ is a commutative associative finitely generated unital algebra over $\mathbb C$ in \cite{5}. In \cite{5}, A.Savage prove that a simple module for $Vir\otimes B$ with finite dimensional weight spaces is either highest weight module, lowest weight module or uniformly bounded module. Broadly speaking, irreducible modules for $Vir\otimes B$ has same kind of behaviour as $Vir$ modules. Classification of irreducible modules for loop-Witt algebras is a challenging open problem. Our results is one step towards that. In this paper we study irreducible weight modules for $(A\rtimes DerA)\otimes B$ with finite dimensional weight spaces, where $B$ is a commutative associative finitely generated unital algebra over $\mathbb C$. It is natural to ask whether modules of tensor fields exhausts all irreducible modules for $(A\rtimes DerA)\otimes B$ or not. Finally in Theorem \ref{t2.3}, we prove that under some natural conditions modules of tensor fields exhausts all irreducible modules for $(A\rtimes DerA)\otimes B$. \\

This paper has been organised as follows. In section 2, we start with basic definitions and preliminaries. At the end of this section we recall the Theorem \ref{t2.1}, proved in \cite{1} regarding classification of irreducible $A\rtimes DerA$ modules for all $n \geq 2$. In section 3, we define weight modules for $\tau=(A \rtimes DerA)\otimes B$ and construct an irreducible weight module for $\tau$. Then we consider an irreducible weight module $V$ for $\tau$ with finite dimensional weight spaces. In Theorem \ref{t2.2}, we prove that under some natural conditions $V$ becomes an irreducible module for $A\rtimes DerA.$ In section 4, we prove that Theorem \ref{t2.1} holds for $n=1$ too. At the end of this section, in Theorem \ref{t2.3}, we prove our final classification theorem.\\

 We would like to point out one difference between loop of finite dimensional simple Lie algebras and loop of any extension of Witt algebras. For the loop of simple finite dimensional Lie algebras we get irreducible modules as evaluation modules at multiple points, see \cite{8,7,19}. But for the loop of Witt algebras which has been studied in \cite{2,5} and the current paper, irreducible modules comes from a single points evaluation modules. So we conjecture that for the classification of loop-Witt algebras with uniformly bounded weight spaces we get a  single point evaluation module.

\section{Notations and Preliminaries}
\subsection*{(1)} Throughout the paper all the vector spaces, algebras, tensor products are taken over the field of complex numbers $\mathbb{C}$. Let $\mathbb{Z}$, $\mathbb{N}$, $\mathbb{Z}_+$ denote the sets of integers, natural numbers and non-negative integers respectively. Let $\mathbb{C}^n $ denote the direct sum of $n$ copies of the vector space $\mathbb C$. Let $e_i$, $1\leq i \leq n$ be the standard basis of $\mathbb {C}^n$ and $(,)$ be the form on $\mathbb{C}^n$ such that $(e_i,e_j)=\delta_{ij}$, for $1 \leq i,j \leq n$. For any Lie algebra $G$, let $ U (G)$ denote the universal enveloping algebra of $G$.\\

\subsection*{(2)} Let $n \geq 1$ be a fixed natural number and $A =\mathbb{C}[t_1^{\pm1},...,t_{n}^{\pm1}]$ denote the Laurent polynomial algebra. Let us denote $\mathbb{Z}^n =\displaystyle{ \bigoplus_{i=1}^{n}}\mathbb Z e_i $. We denote elements of $\mathbb{Z}^n$ by $m ,r ,s ,t $ etc. For $r=\displaystyle{ \sum_{i=1}^{n}}r_i e_i \in \mathbb{Z}^n$,
 let $t^r=t_1^{r_1}t_2^{r_2}...t_n^{r_n} \in A $ and $D^i(r)=t^r t_i \frac{d}{d t_i}$ be a derivation on $A$. Let $Der A$ be the Lie algebra of derivations on $A$. It is well known that $D^i(r)$ for $1 \leq i \leq n$ and $r \in \mathbb{Z}^n$ forms a basis for $DerA.$ For $u=\displaystyle{ \sum_{i=1}^{n}}u_i e_i \in \mathbb{C}^n $, let $D(u,r)= \displaystyle{ \sum_{i=1}^{n}}u_iD^i(r)$. Then $DerA$ is a Lie algebra with respect to the bracket operation
  \begin{align} 
 \left[ D(u,r),D(v,s) \right]=D(w,r+s),
\end{align}
 where $ w=(u,s)v -(v,r)u $ for all $r,s \in \mathbb{Z}^n$ and $u,v \in \mathbb{C}^n$.\\ 
 
Let $\mathfrak{h}$ be the subspace of $DerA$ spanned by $D^i(0)$ for $1\leq i \leq n.$ Then $\mathfrak{h}$ is a maximal abelian sub-algebra of $DerA$.\\ 

In particular when $n=1$, we consider $A=\mathbb C[t^{\pm 1}]$ and denote $d_n=t^nd$ for all $n \in \mathbb Z,$ where $d=t\frac{d}{dt}.$ Then we have $[d_m,d_n]=(n-m)d_{m+n}$ for all $m,n \in \mathbb Z.$
\subsection*{(3)} 
Let $\mathfrak{g}_1$ and $\mathfrak{g}_2 $ be two Lie algebras and $\phi : \mathfrak{g}_2 \to Der(\mathfrak{g}_1)$ be a Lie algebra homomorphism from $\mathfrak{g}_2$ to the Lie algebra of derivations of $\mathfrak{g}_1$. Then the semi-direct product of the Lie algebras $\mathfrak{g}_1$ and $\mathfrak{g}_2$ is denoted by $\mathfrak{g}_1 \rtimes \mathfrak{g}_2= \{g_1+g_2: g_1\in \mathfrak{g}_1, g_2 \in \mathfrak{g}_2 \}$ with the bracket operation 
\begin{align}
\left[ g_1+g_2,h_1+h_2\right] = \left[g_1,h_1\right] +\left[g_2,h_2\right] +\phi(g_2)h_1 - \phi(h_2)g_1.
\end{align} 
\subsection*{(4)} Consider the Lie algebra $A\rtimes DerA$ with the Lie bracket given by 
\begin{align}
\left[ t^r,t^s\right]=0, \\
\left[ D(u,r), t^s\right]=(u,s)t^{r+s}
\end{align}
for all $r,s \in \mathbb{Z}^n$ and $u \in \mathbb{C}^n$. Note that $A$ is an ideal in $A\rtimes DerA$. Let $ H=\mathbb C \oplus \mathfrak{h}$ be a Cartan sub-algebra of $A\rtimes DerA.$ 

\subsection*{(5)} Let $B$ be a commutative associative finitely generated unital algebra over $\mathbb C$. Consider $\tau = (A \rtimes DerA )\otimes B$ and define a Lie algebra structure on $\tau$ by $$\left[X\otimes b, Y\otimes b'\right]=\left[X,Y\right]\otimes bb',$$ for all $X,Y \in A \rtimes DerA$, $b,b' \in B$. We will identify $H\otimes 1$ with $H$. For convenience, we denote $D(u,r)\otimes b=D(u,r)b$ and $t^r\otimes b=t^rb$ for all $u \in \mathbb C^n$, $r \in \mathbb Z^n$, $b\in B.$  

\subsection*{(6)}
Let $\mathfrak{gl}_n$ be the Lie algebra of $n \times n$ matrices over $\mathbb C$. Then $\mathfrak{gl}_n =\mathfrak{sl}_n \oplus \mathbb C I $, where  $\mathfrak{sl}_n$ be the subalgebra of $\mathfrak{gl}_n$ consisting of trace zero matrices and $I$ be the $n\times n$ identity matrix. Let $V(\mu)$ be a finite dimensional irreducible module for $\mathfrak{sl}_n$, where $\mu$ is a dominant integral weight. Let $I$ act by a scalar $c$ on $V(\mu)$ and denote the resultant finite dimensional irreducible $\mathfrak{gl}_n$ module by $V(\mu,c).$ For $\alpha \in \mathbb{C}^n,$ we denote $F^{\alpha}(\mu,c)=V(\mu,c)\otimes A.$ We first recall $A\rtimes DerA$ modules which have been studied in \cite{1}.\\

\begin{theorem}\label{t2.1}$($Theorem 2.9; \cite{1}$)$
Let $n\geq 2$ and $V$ be an irreducible module for $A \rtimes DerA$ which is also a weight module for $H$ with finite dimensional weight spaces. Also assume\\
$(1)$. $V$ is an $A$-module as an associative algebra and Lie module structure of $A$ comes from associative algebra.\\
$(2)$. $1.v=v$ for all $v \in V$.\\
Then $V \simeq F^{\alpha}(\mu,c)$ for some $\alpha, c, \mu$ and actions of $A\rtimes DerA$ on $F^{\alpha}(\mu,c)$ are given by
\begin{align}
D(u,r)v\otimes t^m={(u,m+\alpha)v\otimes t^{m+r}+\displaystyle{\sum_{i,j}(u_ir_jE_{ji}v)\otimes t^{m+r}}}, \\
t^r v \otimes t^m=v \otimes t^{m+r}. 
\end{align}
 
\end{theorem}
\section{}
\begin{definition}
A module $V$ for $\tau$ is said to be weight module if the action of $H$ on $V$ is diagonalizable, i.e $V$ can be decomposed as $V = \displaystyle{\bigoplus_{\lambda \in H^*} V_\lambda}$, where $H^*=Hom_{\mathbb C}(H, \mathbb C)$ and  $V_\lambda=\{ v \in V: hv=\lambda(h)v $ for all $h \in H \}.$ The space $V_\lambda$ is called the weight space with respect to the weight $\lambda$ of $V$.
\end{definition}
Let $\eta:B \to \mathbb C$ be an algebra homomorphism such that $\eta(1)=1$. Define a $\tau$-module action on $F^\alpha(\mu , c)$ by 
\begin{align}
D(u,r)bv\otimes t^m=\eta(b)\{(u,m+\alpha)v\otimes t^{m+r}+\displaystyle{\sum_{i,j}(u_ir_jE_{ji}v)\otimes t^{m+r}}   \},\\
t^rb.v\otimes t^m=\eta(b)v\otimes t^{m+r}.
 \end{align}
It is easy to see that $F^\alpha(\mu , c)$ is a weight module for $\tau$ with finite dimensional weight spaces $V(\mu)\otimes t^m$.  By Theorem \ref{t2.1},  $F^\alpha(\mu , c)$  is an irreducible module for $\tau$.\\

The main purpose of this paper is to prove that under some natural conditions $F^\alpha(\mu,c)$ exhausts all irreducible weight modules for $\tau$ with finite dimensional weight spaces.\\

Let $V$ be an irreducible weight module for $\tau$ with finite dimensional weight spaces. Also assume that the action of $A \otimes B$ on $V$ is associative and $ 1$ acts as identity on $V$. Throughout this section we consider $V$ as a $\tau$-module with these properties. \\ 

Let us assume that there exists a weight $\lambda \in H^*$ such that $V_\lambda \neq 0$. Then due to irreducibility of $V$ we have $V=\displaystyle{\bigoplus_{m \in \mathbb{Z}^n}}V_{\alpha+m},$ where $V_{\alpha+m}= \{ v \in V: D(u,0)v=(u,m+\alpha)v $ for all $u \in \mathbb{C}^n \}$ and $\alpha=(\lambda(D^1(0),\lambda(D^2(0),....,\lambda(D^n(0))\in \mathbb{C}^n.$  Note that $V_\lambda= V_\alpha $ and $t^k$ is injective, hence $t^k V_\alpha=V_{\alpha+k}$ for all $k \in \mathbb{Z}^n$. In particular, all $V_{\alpha+m}$ have equal dimension. Therefore we can  identify $V$ with $V_\alpha \otimes A$ as vector spaces. Since $1\otimes B$ is central in $\tau$ 
there exists an algebra homomorphism $\psi : B \to \mathbb C$ such that $1\otimes b.v=\psi(b)v$ for all $v \in V$. Let $M$ be the kernel of this homomorphism. Then $M$ is a maximal ideal in $B$, since $\psi(1)=1.$\\

 Let $U$ be the universal enveloping algebra of $\tau$. Then $U=\displaystyle{\bigoplus_{m \in \mathbb{Z}^n}}U_m,$ where $U_m= \{ X \in U: [D(u,0),X]=(u,m)X $ for all $u \in \mathbb{C}^n \}$. Let $L$ be the two sided ideal in $U$ generated by $t^rb.t^sb'-t^{r+s}bb'$ and $t^0 \otimes 1 -1.$ Since $A\otimes B$ acts associatively on $V$ so $L$ acts trivially on $V$ and hence $V$ is an irreducible $U/L$-module. \\

Let us consider the elements $T(u,r,b_1,b_2)=t^{-r}b_1D(u,r)b_2  \in U/L$ for all $r \in \mathbb {Z}^n, u \in \mathbb{C}^n, b_1,b_2 \in B$ and define  $$ D =span\{ T(u,r,b_1,b_2):r \in \mathbb {Z}^n, u \in \mathbb{C}^n, b_1,b_2 \in B  \}. $$

\begin{lemma}\label{l2.1}
 1. $D$ is a Lie algebra with the Lie bracket $$[T(u,r,b_1,b_2),T(v,s,b_3,b_4)]=T(w,r+s,b_1b_3,b_2b_4)-(u,s)T(v,s,b_1b_2b_3,b_4)$$ $$+(v,r)T(u,r,b_1b_3b_4,b_2),$$ where $w=(u,s)v-(v,r)u.$\\
 
2. $V_\alpha$ is an irreducible $D$-module.
\end{lemma}

\begin{proof}
1. $[t^{-r}b_1D(u,r)b_2, t^{-s}b_3D(v,s)b_4]$\\

$=[t^{-r}b_1D(u,r)b_2, t^{-s}b_3]D(v,s)b_4+t^{-s}b_3[t^{-r}b_1D(u,r)b_2, D(v,s)b_4]$\\

$=-(u,s)t^{-s}b_1b_2b_3D(v,s)b_4+(v,r)t^{-r}b_1b_4b_3D(u,r)b_2+t^{-r-s}b_1b_3D(w,r+s)b_2b_4$\\

$=T(w,r+s,b_1b_3,b_2b_4)-(u,s)T(v,s,b_1b_2b_3,b_4)+(v,r)T(u,r,b_1b_3b_4,b_2),$\\

 where $w=(u,s)v-(v,r)u.$\\

2. Check that $D(w,0)$ commutes with $D$ for all $w \in \mathbb{C}^n$, consequently $V_\alpha$ is a $D$-module. 
 Let $v,w \in V_\alpha$ be two arbitrary non-zero vectors. Then by weight arguments and the irreducibility of $V$ implies that there exists $X \in U_0$ such that $X.v=w$.
By the PBW theorem we have $U(D)=U_0$ and hence $V_\alpha$ is an irreducible $D$-module.
\end{proof}

\begin{lemma}$($ \cite{3}$)$\label{l2.2}
1. Let $L \subset \mathfrak{gl}(V)$ $(V$ is finite dimensional $)$ be a non-zero Lie algebra and acting irreducibly on $V$.  Then $L$ is reductive and its center is at most one-dimensional. \\
2. Let $L$ be a finite dimensional reductive Lie algebra. Then $L=[L,L]\oplus Z(L)$ and $[L,L]$ is semi-simple.
\end{lemma}

Let us consider the elements $T_1(u,r,b_1,b_2)=t^{-r}b_1D(u,r)b_2 -D(u,0)b_1b_2$ and define 
  $$D_1=span\{ T_1(u,r,b_1,b_2):r \in \mathbb {Z}^n, u \in \mathbb{C}^n, b_1,b_2 \in B  \}.$$
Also let $I(u,r,b_1,b_2)= \psi(b_1)D(u,r)b_2-D(u,0)b_1b_2$ and define  $$D_2 =span\{ I(u,r,b_1,b_2):r \in \mathbb {Z}^n, u \in \mathbb{C}^n, b_1,b_2 \in B  \}. $$
Moreover consider the subspace $W$ of $V,$ defined by $$W= span\{t^rv-v: v\in V_\alpha, r \in \mathbb{Z}^n\}=span\{t^rv-v: v\in V, r \in \mathbb{Z}^n\}$$  Note that $dim(V/W)=dim(V_\alpha)$.\\

\begin{proposition}\label{p2.1}

1. $D_1$ is a Lie algebra with the 
lie bracket
 $$[T_1(u,r,b_1,b_2),T_1(v,s,b_3,b_4)]= T_1(w,r+s,b_1b_3,b_2b_4)-(u,s)T_1(v,s,b_3,b_1b_2b_4)$$ $$+(v,r)T_1(u,r,b_1,b_2b_3b_4),$$ where $w=(u,s)v-(v,r)u$. Moreover $[D,D]\subseteq D_1.$\\
 
2. $D_2$ is a Lie sub-algebra of $DerA \otimes B$ with the Lie bracket $$[I(u,r,b_1,b_2),I(v,s,b_3,b_4)]= I(w,r+s,b_1b_3,b_2b_4)-(u,s)I(v,s,b_3,b_1b_2b_4)$$  $$+(v,r)I(u,r,b_1,b_2b_3b_4),$$ where $w=(u,s)v-(v,r)u$.\\
3. $\pi: D_1 \to D_2$ defined by $\pi(T_1(u,r,b_1,b_2)=I(u,r,b_1,b_2)$ is a surjective Lie algebra homomorphism.\\
$4$. $V_\alpha$ is an irreducible $D_1$-module.\\
5. $W$ is a $D_2$-sub-module. In particular $V/W$ is a $D_2$-module and a $D_1$-module via $\pi$. \\
6. $V_\alpha \simeq V/W$ as $D_1$-module.
\end{proposition}
\begin{proof}
1. From Lemma \ref{l2.1} we have\\

$[T(u,r,b_1,b_2),T(v,s,b_3,b_4)]$\\

$=T(w,r+s,b_1b_3,b_2b_4)-(u,s)T(v,s,b_1b_2b_3,b_4)+(v,r)T(u,r,b_1b_3b_4,b_2)$\\

$=\{T(w,r+s,b_1b_3,b_2b_4)-D(w,0)b_1b_2b_3b_4\} -(u,s)\{T(v,s,b_1b_2b_3,b_4)-D(v,0)b_1b_2b_3b_4\}$ \\

 \hspace*{3cm} $+(v,r)\{T(u,r,b_1b_3b_4,b_2)-D(u,0)b_1b_2b_3b_4\}  $\\

$=T_1(w,r+s,b_1b_3,b_2b_4)-(u,s)T_1(v,s,b_1b_2b_3,b_4)+(v,r)T_1(u,r,b_1b_3b_4,b_2)$.\\

From above relation we have $[D,D]\subseteq D_1.$\\

Now $[T_1(u,r,b_1,b_2),T_1(v,s,b_3,b_4)]$\\

$=[t^{-r}b_1D(u,r)b_2,t^{-s}b_3D(v,s)b_4]-[D(u,0)b_1b_2,t^{-s}b_3D(v,s)b_4]+[D(v,0)b_3b_4,t^{-r}b_1D(u,r)b_2]$\\

$=[T(u,r,b_1b_2), T(v,s,b_3b_4)]-\{ [D(u,0)b_1b_2,t^{-s}b_3]D(v,s)b_4+t^{-s}b_3[D(u,0)b_1b_2,D(v,s)b_4]\}$ \\

   \hspace*{3cm} $+\{[D(v,0)b_3b_4,t^{-r}b_1]D(u,r)b_2 + t^{-r}b_1[D(v,0)b_3b_4,D(u,r)b_2]    \}$\\

$=[T(u,r,b_1b_2), T(v,s,b_3b_4)]+(u,s)T(v,s,b_1b_2b_3,b_4)-(u,s)T(v,s,b_3,b_1b_2b_4) $ \\ 

\hspace*{3cm} $-(v,r)T(u,r,b_1b_3b_4,b_2)+(v,r)T(u,r,b_1,b_2b_3b_4)$\\

$=T(w,r+s,b_1b_3,b_2b_4)-(u,s)T(v,s,b_3,b_1b_2b_4)+(v,r)T(u,r,b_1,b_2b_3b_4)$ (using Lemma \ref{l2.1}(1)) \\

$=T_1(w,r+s,b_1b_3,b_2b_4)-(u,s)T_1(v,s,b_3,b_1b_2b_4)+(v,r)T_1(u,r,b_1,b_2b_3b_4).$\\

2. $[I(u,r,b_1,b_2),I(v,s,b_3,b_4)]$\\

$=[\psi(b_1)D(u,r)b_2-D(u,0)b_1b_2,\psi(b_3)D(v,0)b_4-D(v,0)b_3b_4]$\\

$=\psi(b_1b_3)[D(u,r)b_2,D(v,s)b_4]+(v,r)\psi(b_1)D(u,r)b_2b_3b_4-(u,s)\psi(b_3)D(v,s)b_1b_2b_4$\\

$= \psi(b_1b_3)D(w,r+s)b_2b_4+(v,r)\psi(b_1)D(u,r)b_2b_3b_4-(u,s)\psi(b_3)D(v,s)b_1b_2b_4$\\

$=I(w,r+s,b_1b_3,b_2b_4)+(v,r)I(u,r,b_1,b_2b_3b_4)-(u,s)I(v,s,b_3,b_1b_2b_4)$.\\

3. Follows from 1 and 2.\\
 
$4$. Check that $D(w,0)$ commutes with $D_1$ for all $w \in \mathbb{C}^n$. Therefore $V_\alpha$ is a $D_1$-module. By Lemma \ref{l2.1}, we have $V_\alpha$ is an irreducible $D$-module. Let $\phi: D \to End(V_\alpha)$ be the representation. Then by Lemma \ref{l2.2}, we have $\phi(D)$ is reductive and $\phi(D)=Z(\phi(D)) \oplus \phi([D,D]).$ Since $V_\alpha$ is irreducible,  $Z(\phi(D))$ acts as scalar on $V_\alpha$. Consequently $V_\alpha$ is an irreducible module for $[D,D]\subset D_1.$ \\

5. Since the action of $A\otimes B$ is associative, so $t^sb.v=t^s\otimes 1 . 1 \otimes b.v=\psi(b)t^sv$, for all $s\in \mathbb Z^n, b\in B$ and $v \in V.$ \\

Now, $(\psi(b_1)D(u,r)b_2-D(u,0)b_1b_2).(t^sv-v)$\\

$=\psi(b_1)[t^s(D(u,r)b_2.v+(u,s)t^{s+r}b_2.v-D(u,r)b_2.v] -[t^sD(u,0)b_1b_2.v+(u.s)t^sb_1b_2.v-D(u,0)b_1b_2.v]
$\\

$= \psi(b_1)[t^s(D(u,r)b_2.v-D(u,r)b_2.v]+(u,s)\psi(b_1)[t^{s+r}b_2.v-t^sb_2.v]- [t^sD(u,0)b_1b_2.v-D(u,0)b_1b_2.v] \in W
$.\\

6. Define $\phi: V \to V/W$ by  $\phi(v)=\overline v 
$ for all $v \in V$. Now restrict $\phi$ to $V_\alpha$ and denote by $\phi|_{V_\alpha}$. Note that $\overline{t^sv}=\overline v$ for all $s \in \mathbb Z^n, v \in V.$\\

Claim: $\phi(T_1(u,r,b_1,b_2).v)=\pi(T_1(u,r,b_1,b_2)).\overline v,$ for all $v \in V_\alpha, u\in \mathbb C^n, r\in \mathbb Z^n, b_1,b_2 \in B$.\\
Now,
$\phi(T_1(u,r,b_1,b_2).v)$\\

 $ =\overline{ t^{-r}b_1D(u,r)b_2.v -D(u,0)b_1b_2.v}$ \\
 
 $=\overline{\psi(b_1)t^{-r}D(u,r)b_2.v} -\overline{D(u,0)b_1b_2.v}$ \\
 
 $=\overline{\psi(b_1)D(u,r)b_2.v} -\overline{D(u,0)b_1b_2.v}$ (Since $\overline{t^sv}=\overline v$)\\
 
 $=\overline{\psi(b_1)D(u,r)b_2.v-D(u,0)b_1b_2.v}$\\
 
 $=\pi(T_1(u,r,b_1,b_2)).\overline v$.\\
 
Thus $\phi|_{V_\alpha}$ is a $D_1$-module map.
Now $\phi|_{V_\alpha}$ is non-zero and $V_\alpha$ is an irreducible $D_1$-module, so $\phi|_{V_\alpha}$ is injective. $\phi|_{V_\alpha}$ is surjective because of equal dimension and hence $V_\alpha \simeq V/W$ as $D_1$-module.
\end{proof}

Let us consider the space $$\widetilde D=span\{\psi(b_1)D(u,r)b_2-D(u,r)b_1b_2:r \in \mathbb {Z}^n, u \in \mathbb{C}^n, b_1,b_2 \in B\}.$$ Observe that $\widetilde D \subset D_2$ and $\widetilde D=DerA \otimes M$, where $M$ is the maximal ideal $ker\psi$. Therefore $\widetilde D$ is an ideal in $DerA \otimes B$ and hence an ideal in $D_2$. 

\begin{lemma}\label{l2.3}

 $D(u,0)b$ acts as scalar on $V_\alpha$ for all $u \in \mathbb C^n, b \in B.$
\end{lemma}
\begin{proof}

By Proposition \ref{p2.1}, $D_2$ acts on $V/W$ irreducibly. Let $\phi: D_2\to End(V/W)$ be the finite dimensional  irreducible representation and $\rho=\phi|_{\widetilde D}.$ We prove this lemma in two cases. \\
 Case 1: Let $M^k=0$ for some $k \in \mathbb N$.\\
We have $[DerA\otimes M^i, DerA\otimes M^j]=DerA\otimes M^{i+j}$ for all $i,j \in \mathbb N$, since $DerA$ is a simple Lie algebra. This implies that $DerA\otimes M$ is a solvable ideal in $D_2$, thus $\phi(\widetilde D)$ is a solvable ideal in the Lie algebra $\phi(D_2)$. By Lemma \ref{l2.2}, $\phi(D_2)$ is a reductive Lie algebra and hence $\phi(\widetilde D)$ acts as a scalar on $V/W$. In particular, $\psi(b)D(u,0)-D(u,0)b$ acts as scalar on $V/W.$ Now by Proposition \ref{p2.1}(6,3), we have $T_1(u,0,b,1)$ acts as scalar on $V_\alpha$. This gives us $D(u,0)b$ acts as scalar on $V_\alpha$.\\
 
Case 2: Let $M^k \neq 0$ for all $k \in \mathbb N.$ In this case, it is sufficient to show that $DerA\otimes M^k
 $ acts trivially on $V/W$ for some $k \in \mathbb N$. Because $[ \frac{ DerA\otimes M^i}{ker \phi}, \frac{ DerA\otimes M^j}{ker \phi}]= \frac{ DerA\otimes M^{i+j}}{ker \phi}$ for all $i,j \in \mathbb N$ implies that  $\frac{ DerA\otimes M^2}{ker \phi}$ is a solvable ideal contained in the semi-simple Lie algebra $[\frac{D_2}{ker \phi}, \frac{D_2}{ker \phi}].$ Thus $DerA \otimes M^2$ acts trivially on $V/W.$ Therefore
 $\phi(DerA\otimes M)$ is an abelian ideal in $\phi(D_2)$. Now the result follows by same argument as case 1.\\
 
\textbf{Claim:} $DerA\otimes M^{2N + 6} .V/W=0$, where $N=$ dim $(DerA\otimes M/J)$, $J=ker \rho$. \\

Sub-claim 1: $D(v,0)\otimes M^{2N+5}\subset J$ for all $v \in \mathbb{C}^n,$ for all $n \geq 2.$\\

 Let $b=b_1b_2..,b_{2N+5}$ be a non-zero element of $M^{2N + 5}$. Consider the set of vectors $$\{ D(u_i,r_i)b_1 +J:u_i\in \mathbb{C}^n- \{0\}, r_i \in \mathbb{Z}^n -\{0\}, r_i\neq r_j , \forall i\neq j, 1 \leq i\leq N+1\}.$$ Since  dim $(DerA\otimes M/ker \rho)=N,$ there exists a non-zero vector $X=\displaystyle{\sum_{i=1}^{k}D(u_i,r_i)b_1}\in J$ for some $k \leq N+1.$\\ Now,
 $$[X,D(v,-r_1)b_2]=D(w_1,0)b_1b_2 +\displaystyle{\sum_{i=2}^{k}}D(w_i,r_i-r_1)b_1b_2 ,$$ where $w_1=-(u_1,r_1)v-(v,r_1)u_1$ and $w_i=-(u_i,r_1)v-(v,r_i)u_i$ for all $i\neq 1.$\\

 Thus if $(u_1,r_1)=0$ choose $v$ such that $(v,r_1)\neq 0$ and if $(u_1,r_1) \neq 0$ choose $v$ such that $(v,r_1)= 0$, in both cases $w_1\neq 0.$ If all $w_i=0$ for all $i\neq 1,$ then we have $D(w_1,0)b_1b_2 \in J.$
If $w_i\neq 0$ for some $i\neq 1$, say  $w_j\neq 0$. Then choosing any $v'$ such that $(v',r_j-r_1)\neq 0$ and considering $[[X,D(v,-r_1)b_2],D(v',0)b_3]$ we get the sum of at most $k-1$ terms. Continuing the above process we can conclude that $D(u,0)b_1b_2..b_{2k+1} \in J $ for some non-zero $u \in \mathbb{C}^n. $\\

For any non-zero $v \in \mathbb{C}^n$ choose  $s \in \mathbb{Z}^n$ such that $(u,s)\neq 0$. Then we have

 $$[D(v,s)b_{2k+2},D(u,0)b_1b_2..b_{2k+1}]= (u,s)D(v,s)b_1b_2...b_{2k+2} \in J.$$  
Again, $$[D(w,-s)b_{2k+3},D(v,s)b_1b_2...b_{2k+2}]=D((w,s)v+(v,s)w,0)b_1...b_{2k+3}\in J .$$
Now, if $(v,s)=0$ choose $(w,s)\neq 0$ and if $(v,s)\neq 0$ choose $w =v $, then with this choice we have for all $v(\neq 0) \in \mathbb{C}^n$, $D(v,0)b_1b_2....b_{2k+3} \in J$. Since $k \leq N+1$ we have $D(v,0)\otimes M^{2N+5}\subset J$ for all $v \in \mathbb{C}^n.$\\
 
Sub-claim 2: $d_0\otimes M^{2N+5} \subset J$ for $n=1.$\\

Let $b=b_1b_2...b_{2N+5} \in M^{2N+5}$. Let us consider the set of vectors\\

 $\{ d_{r_i}b_1 +J : r_i \in \mathbb N, r_i>r_j $ for $i>j$  and $ 2r_i \neq r_j +r_k$ for distinct $i,j, k , 1 \leq i \leq N+1\}$\\ 
 
 ( In particular $2^i , i \in \mathbb N$ can be chosen as $r_i)$. Since dim$(DerA \otimes M)/ J=N$, there exists a non-zero vector $X= \displaystyle{\sum_{i=1}^{k}\lambda _id_{r_i}b_1}\in J$ with some  $\lambda_i( \neq 0) \in \mathbb C$, $k \leq 
N+1$. Let us assume $\lambda_1 \neq 0$ and  consider

 $$[d_{-r_1}b_2,X]=2r_1\lambda_1d_0b_1b_2+\displaystyle{\sum_{i=2}^{k}\lambda _i(r_i+r_1)d_{r_i-r_1}b_1b_2}\in J.$$ 
 
If $\lambda_i =0 $ for all $i\geq 2$ we are done. If not, let $\lambda_2 \neq 0$ and consider 

$$[d_{r_1-r_2}b_4,[d_0b_3,[d_{-r_1}b_2,X]]]=2(r_2-r_1)^2(r_2+r_1)\lambda_2d_0b_1b_2b_3b_4$$  $$ + \displaystyle{\sum_{i=3}^{k}\lambda _i(r_i^2-r_1^2)(r_i+r_2-2r_1)d_{r_i-r_2}b_1b_2b_3b_4}\in J.$$ Continuing this process we can conclude that $d_0 \otimes M^{2N+5} \subset J$, since $k \leq N+1.$
Therefore sub-claim 1 is true for all $n \geq 1.$\\

 Now for any $b \in M$, $b'\in M^{2N+5}$, $u (\neq 0) \in \mathbb{C}^n $ and $s (\neq 0) \in \mathbb{Z}^n$, choose $v $ such that $(v,s)\neq 0$ and then consider $[D(u,s)b,D(v,0)b'] $ which gives $D(u,s)bb'\in J. $ Hence the claim follows.
\end{proof}

Let $T(u,r)b=T(u,r,1,b)$ for all $u,r,b$ and consider $T_B$ be the space spanned by $T(u,r)b.$ It is easy to see that $T_B$ is a sub-algebra of $D_1.$ In fact one can check that 
$$[T(v,s)b,T(u,r)b'] = T(w,r+s)bb' +(u,s)T(v,s)bb'-(v,r)T(u,r)bb', $$
where $w=(v,r)u-(u,s)v.$
\begin{theorem}\label{t2.2}
Let $V$ be an irreducible weight module for $\tau$ with finite dimensional weight spaces. Also let the action of $A\otimes B$ on $V$ be associative and $1.v=v$ for all $v \in V.$ Then $V$ is an irreducible module for $A\rtimes DerA.$
\end{theorem}

\begin{proof}
By Lemma \ref{l2.3}, $D(u,0)b $ acts on $V_\alpha$ as a scalar, say $\lambda(u,b)$ for all $u,b$. Then we have,
\begin{align}\label{a2.9}
D(u,0)b(v_\alpha\otimes t^m)=(\lambda(u,b)+(u,m)\psi(b))v_\alpha \otimes t^m\\
 t^mb.(v_\alpha \otimes t^k)=\psi(b)v_\alpha \otimes t^{m+k}=\psi(b)t^m.(v_\alpha \otimes t^k).
\end{align} 
for all $m, k \in \mathbb Z^n$ and $v_\alpha \in V_\alpha.$\\  Consider the actions of $T(u,r)b$ on $V$, for all $r \neq 0, u ,b.$ 
 $$T(u,r)b.v\otimes t^k= (t^{-r}D(u,r)b-D(u,0)b)v\otimes t^k.$$ Now using (\ref{a2.9}) and the fact that $t^r$ commutes with $T(u,r)b$ we have,
 \begin{align}\label{a2.11}
 D(u,r)b.v\otimes t^k=T(u,r)b.v\otimes t^{k+r}+(\lambda(u,b)+(u,k)\psi(b))v\otimes t^{k+r}
\end{align}  
for all $v \in V_\alpha$ and $k \in \mathbb Z^n$. \\
 In particular for $b=1$, (\ref{a2.11}) gives
 \begin{align}\label{a2.12}
 D(u,r).v\otimes t^k=T(u,r).v\otimes t^{k+r}+(u,k+ \alpha)\psi(b)v\otimes t^{k+r}, 
 \end{align}
since $\lambda(u,1)=(u,\alpha)$.\\
For any $r\neq 0$ choose $v$ such that $(v,r)\neq 0$ and consider 
\begin{align}\label{a2.13}
[D(v,0)b, D(u,r)]=(v,r)D(u,r)b.
\end{align} 
Then using the actions of (\ref{a2.9}), (\ref{a2.11}) and (\ref{a2.12}) on the both sides of (\ref{a2.13}) we have the following action of $T(u,r)b $ on $V$, 
\begin{align}\label{a*}
T(u,r)b= \psi(b)T(u,r)+ \psi(b)(u,\alpha)-\lambda(u,b)
\end{align} 
for all $r \neq 0.$\\ 
Now using (\ref{a2.11}), (\ref{a2.12}) and (\ref{a*}) it is immediate to check that
\begin{align}\label{a2.15}
D(u,r)b=\psi(b)D(u,r)    
\end{align} 
on $V$ for all $r \neq 0.$\\
For all non-zero $u,v,r,s $ and $b$ we have
\begin{align}\label{a2.16}
[D(v,s)b,D(u,r)]w_0\otimes t^k=D((v,r)u-(u,s)v,r+s)bw_0\otimes t^k,
\end{align}
for all $w_0 \in V_\alpha$ and $k \in \mathbb Z^n$.\\
Using (\ref{a2.15}) in the left hand side of (\ref{a2.16}) we have
\begin{align}\label{a2.17}
\psi(b)D((v,r)u-(u,s)v,r+s)w_0\otimes t^k=D((v,r)u-(u,s)v,r+s)b.w_0\otimes t^k.
\end{align} 
In (\ref{a2.17}), put $u=v $ and find $r$ such that $(u,r)\neq 0$. Then put $s=-r$ and conclude that $D(u,0)b=\psi(b)D(u,0)$ on $V$. Hence $\lambda(u,b)=\psi(b)(u,\alpha)$ and $D(u,r)b=\psi(b)D(u,r) $ on $V$ for all $u,r$. This completes the proof.

\end{proof}

\begin{remark}
Theorem \ref{t2.2} shows that our module for $\tau$ is actually an evaluation module at a single point. For more details on evaluation module see \cite{2,5,19} and references therein.
\end{remark}

\section{}
In \cite{1}, Theorem \ref{t2.1} was proved for $A=\mathbb C[t_1^{\pm 1},t_2^{\pm 1},....,t_n^{\pm 1}] $ with the assumption $n \geq 2$. In this section we show that the Theorem \ref{t2.1} holds for  $A=\mathbb C[t^{\pm 1}]$ too.\\

Let $V$ be a module for $A \rtimes DerA$ satisfying all properties of Theorem \ref{t2.1}. Let us assume that there exists $\lambda \in H^*$ such that $V_\lambda \neq 0$. Then $V=\displaystyle{\bigoplus_{m \in \mathbb{Z}}}V_{m+\alpha}$,\\ where $V_{m+\alpha}= \{ v \in V: d_0v=(m+\alpha)v \}$, $\alpha=\lambda(d)\in \mathbb{C}.$ Note that $V_\lambda= V_\alpha $ and $t^k$ is injective, hence $t^k V_\alpha=V_{k}$ for all $k \in \mathbb{Z}$. Therefore we can identify $V$ with $V_\alpha \otimes A $ as vector space.
\\ 

Let $I(r)=(t^r-1)d$, for all $r \in \mathbb Z$ and denote $T'=span\{I(r): r \in \mathbb Z\}$. Let $\textbf{ m}$ be the maximal ideal in $A$ generated by $(t-1)$. It is clear that $\textbf{m}d=T'$.
\begin{lemma}\label{l2.4}
 $[I(r),I(s)]=(s-r)I(r+s)+rI(r)-sI(s)$, for all $s,r\in \mathbb Z.$ In particular, $T'$ is a sub-algebra of $A\rtimes DerA$.
\end{lemma}
Let  $W=span\{t^r.v-v: v \in V_\alpha$ and $r \in \mathbb{Z}\}$.
 It is easy to see that $W$ is a $T'$ sub-module of $V$ and hence $\overline V= V/W$ is a $T'$ module. Now we define an action of $A\rtimes DerA$ on $\overline V \otimes A$ by
\begin{align}\label{a2.18}
d_r.v\otimes t^s=(I(r).v)\otimes t^{r+s}+(s+\alpha)v\otimes t^{r+s},\\
t^r.v\otimes t^s= v\otimes t^{r+s},
\end{align} 
for all $r, s \in \mathbb{Z}$, $v \in \overline V$ and for some $\alpha \in \mathbb{C}$. It is easy to see that this action define a module structure on $\overline{V}\otimes A.$ We will denote this module by $\overline V^\alpha \otimes A$.
\begin{proposition}\label{p2.2}
$V \simeq \overline V ^\alpha \otimes A$ as $A \rtimes DerA$-module and $\overline V ^\alpha$ is a finite dimensional irreducible $T'$-module.
\end{proposition}
\begin{proof}
 Define $\phi: V \to \overline V^\alpha \otimes A$ by $\phi(t^mv)=\overline{t^mv}\otimes t^m$ for all $v \in V_\alpha$ and $m \in \mathbb Z$. Then $\phi $ is a $A \rtimes DerA$ module homomorphism. $\phi$ is clearly surjective. Since $\phi(V_\alpha)=\overline V^\alpha$, hence $\phi$ is non-zero. So by irreducibility of $V$ we have $\phi$ is injective. \\
Suppose $\overline W$ be a $T' $ sub-module of $\overline V^\alpha$, then $\overline W \otimes A$ becomes a $A\rtimes DerA$ sub-module of $\overline V^\alpha \otimes A$, a contradiction. 
\end{proof}

\begin{lemma}\label{l2.5}

1. $$\left[ (t-1)^kd_i,(t-1)^ld_j \right]=(l-k+j-i)(t-1)^{k+l}d_{i+j} + (l-k)(t-1)^{k+l-1}d_{i+j},$$
for all $k,l \in \mathbb N, i,j \in \mathbb Z.$\\

2. For any polynomial $f(\lambda)$, $$f((1-t^{-1})d_{0}-k).(t-1)^{k+1}d_{-1}.v=(t-1)^{k+1}d_{-1}.f((1-t^{-1})d_{0}).v,$$ for all $v \in V$ and $k \in \mathbb{N}$.
\end{lemma}
\begin{proof}
1. Follows from Lemma 2.4 of \cite{4}.\\

2. Note that, $((1-t^{-1})d_{0}-k).(t-1)^{k+1}d_{-1}.v$\\

$= \left[(1-t^{-1})d_{0},(t-1)^{k+1}d_{-1}\right].v+(t-1)^{k+1}d_{-1}.(1-t^{-1})d_{0}.v-k(t-1)^{k+1}d_{-1}.v$\\

$=\left[(t-1)d_{-1},(t-1)^{k+1}d_{-1}\right].v+(t-1)^{k+1}d_{-1}.(1-t^{-1})d_{0}.v-k(t-1)^{k+1}d_{-1}.v$\\

$=k(t-1)^{k+2}d_{-2}.v+k(t-1)^{k+1}d_{-2}.v + (t-1)^{k+1}d_{-1}.(1-t^{-1})d_{0}.v-k(t-1)^{k+1}d_{-1}.v$\\

$=k(t-1)^{k+1}[(t-1)t^{-1}+t^{-1}]d_{-1}.v +(t-1)^{k+1}d_{-1}.(1-t^{-1})d_{0}.v-k(t-1)^{k+1}d_{-1}.v$\\

$=(t-1)^{k+1}d_{-1}.(1-t^{-1})d_{0}.v$\\

Now assume that, 
\begin{align}\label{a2.20}
((1-t^{-1})d_{0}-k)^n.(t-1)^{k+1}d_{-1}.v=(t-1)^{k+1}d_{-1}.((1-t^{-1})d_{0})^n.v
,\end{align}
for some $n \in \mathbb N.$
 Then we have,\\

$((1-t^{-1})d_{0}-k)^{n+1}.(t-1)^{k+1}d_{-1}.v$\\ 

$=((1-t^{-1})d_{0}-k).[(t-1)^{k+1}d_{-1}.((1-t^{-1})d_{0})^n.v], $ by (\ref{a2.20})\\

$=(t-1)^{k+1}d_{-1}.(1-t^{-1})d_{0}.((1-t^{-1})d_{0})^n.v$\\

$=(t-1)^{k+1}d_{-1}.((1-t^{-1})d_{0})^{n+1}.v$\\

 Hence by induction we have the result for any polynomial $f(\lambda)$.
\end{proof}
\begin{proposition}\label{p2.3}

1.  $\textbf{m}^2d\overline V^\alpha=0$.\\
2. $V$ is an irreducible module for  $\mathbb C$ and satisfy the actions given in Theorem \ref{t2.1}.
\end{proposition}
\begin{proof}
1. It is sufficient to prove that $\textbf{m}^kd\overline V^\alpha=0$ for some $k \in \mathbb N.$ Because then using same argument as Lemma \ref{l2.3} (Case 2) we have the result. Consider the operator $(1-t^{-1})d_{0}: \overline V^\alpha \to \overline V^\alpha$. Let $f(\lambda) $ be the characteristic polynomial of this operator. Since we are over the field $\mathbb C$, there exists a $k\in \mathbb N$ such that $gcd(f(\lambda), f(\lambda -l))=1$ for all $l> k$. Now using Lemma \ref{l2.5}(2), we get $(t-1)^{l+1}d_{-1}.v=0$ for all $l > k$ and $v \in \overline V^\alpha.$ 
By Lemma \ref{l2.5}(1), we have the relation $$[(t-1)^{l+2}d_{-1},(t-1)d_{i}]-[(t-1)^{l+1}d_{-1},(t-1)^2d_{i}]=-2(t-1)^{l+2}d_i$$ for all $i \in \mathbb Z,l\in \mathbb N.$ This implies that $(t-1)^{l+2}t^id.v=0$ for all $v\in \overline V^\alpha$, $i \in \mathbb Z,$ $l >k.$ Therefore there exists $k_0 \in \mathbb N$ such that $\textbf{m}^{k_0}d\overline V^\alpha=0.$\\

2. Let us define a map $\phi:{\textbf{m}d }/{\textbf{m}^2d} \to \mathbb C$ by $\phi((t-1)f(t)d)= f(1)$. It is easy to see that $\phi$ is a Lie algebra isomorphism and image of $I(r)$ under this map is $r$ for all $r \in \mathbb Z.$ Now by Proposition \ref{p2.2} and the action defined by (\ref{a2.18}), we have
$$ d_r.v\otimes t^s=r(1.v)\otimes t^{r+s}+(s+\alpha)v\otimes t^{r+s} $$

\end{proof}

\begin{theorem}\label{t2.3}
Let $V$ be an irreducible weight module for $\tau$ with finite dimensional weight spaces. Also let the action of $A\otimes B$ on $V$ be associative and $1.v=v$ for all $v \in V.$ Then $V \simeq F^\alpha(\mu,c)$ for some dominant integral weight $\mu$, $c \in \mathbb C$ and $\alpha \in \mathbb C^n$. Actions of the elements of $\tau$ on $F^\alpha(\mu,c)$ are given by 
$$D(u,r)bv\otimes t^m=\psi(b)\{(u,m+\alpha)v\otimes t^{m+r}+\displaystyle{\sum_{i,j}(u_ir_jE_{ji}v)\otimes t^{m+r}}   \},$$ 

$$ t^rb.v\otimes t^m=\psi(b)v\otimes t^{r+m},$$  \\
 for all $u \in \mathbb C^n, r \in \mathbb Z^n, b \in B$ and for some algebra homomorphism $\psi: B \to \mathbb C.$ \\
  
\end{theorem}
\begin{proof}
Follows from Theorem \ref{t2.1}, Theorem \ref{t2.2}, Proposition \ref{p2.2} and Proposition \ref{p2.3}.
\end{proof}

\vspace{2cm}
\textbf{Acknowledgments:} The first author would like to thank Prof. Rencai Lu for some helpful discussion regarding one variable case of $(A,DerA)$.

\end{document}